\documentclass[12pt,a4paper,reqno]{amsart}
\usepackage[T1]{fontenc}
\usepackage[utf8]{inputenc}
\usepackage{lmodern}
\usepackage{microtype}
\usepackage{amsmath,amssymb,amsthm,mathtools}
\usepackage{algorithm}
\usepackage{algpseudocode}
\usepackage[pdftex,pdfpagelabels]{hyperref}
\hypersetup{
  hidelinks,
  pdftitle={On Supersingular Isogeny Graphs of Drinfeld Modules},
  pdfauthor={Nikola Veselinov}
}
\usepackage[nameinlink,noabbrev]{cleveref}
\usepackage{xcolor}
\usepackage{enumitem}
\usepackage{booktabs}

\usepackage{comment}

\numberwithin{equation}{section}

\usepackage{aliascnt}

\newtheorem{theorem}{Theorem}[section]

\newaliascnt{proposition}{theorem}
\newtheorem{prop}[proposition]{Proposition}
\aliascntresetthe{proposition}

\newaliascnt{lemma}{theorem}
\newtheorem{lemma}[lemma]{Lemma}
\aliascntresetthe{lemma}

\newaliascnt{corollary}{theorem}

\aliascntresetthe{corollary}

\newaliascnt{example}{theorem}

\aliascntresetthe{example}

\newaliascnt{conjecture}{theorem}

\aliascntresetthe{conjecture}

\theoremstyle{definition}
\newaliascnt{defn}{theorem}
\newtheorem{defn}[defn]{Definition}
\crefname{corollary}{Corollary}{Corollaries}
\aliascntresetthe{defn}

\theoremstyle{remark}
\newaliascnt{remark}{theorem}
\newtheorem{remark}[remark]{Remark}
\aliascntresetthe{remark}

\newcommand\Z{\mathbb{Z}}
\newcommand\C{\mathbb{C}}
\newcommand\F{\mathbb{F}}

\newcommand\PP{\mathbb{P}}

\renewcommand\O{\mathcal{O}}

\newcommand{\1}{\mathbf{1}}

\newcommand{\pp}{\mathfrak{p}}
\renewcommand{\ss}{\mathfrak{s}}
\newcommand{\qq}{\mathfrak{q}}
\newcommand{\rr}{\mathfrak{r}}
\newcommand{\nn}{\mathfrak{N}}
\newcommand{\mm}{\mathfrak{m}}

\newcommand{\N}{\mathcal{N}}
\renewcommand{\P}{\mathcal{P}}

\newcommand{\ch}{\operatorname{char}}
\renewcommand{\hom}{\operatorname{Hom}}
\newcommand{\End}{\operatorname{End}}
\newcommand{\nrd}{\operatorname{Nrd}}
\newcommand{\cl}{\operatorname{Cl}}
\newcommand{\length}{\operatorname{length}}
\newcommand{\height}{\operatorname{ht}}
\newcommand{\gl}{\operatorname{GL}}
\newcommand{\mat}{\operatorname{Mat}}
\newcommand{\aut}{\operatorname{Aut}}
\newcommand{\gal}{\operatorname{Gal}}
\newcommand{\ov}{\overline}
\newcommand{\diag}{\operatorname{diag}}
\newcommand{\supp}{\operatorname{supp}}

\title{On Supersingular Isogeny Graphs of Drinfeld Modules}
\author{Nikola Veselinov}
\address{Independent researcher}
\email{nikola.veselinov.veselinov@gmail.com}
\subjclass[2020]{Primary 11G09; Secondary 05C25, 11R52}
\keywords{supersingular Drinfeld modules, isogeny graphs, Brandt matrices, Bruhat--Tits trees, quaternion algebras over function fields, completeness number, Drinfeld modular polynomials}
\date{August 2026}

\begin{document}

\begin{abstract}
    We study supersingular isogeny graphs of rank-two Drinfeld modules over $A=\F_q[T]$. For distinct finite primes $\pp$ and $\qq$ of $A$, we prove that the graph in characteristic $\pp$, with edges given by cyclic $\qq$-isogenies, is connected. The proof combines Gekeler's ideal-class correspondence with strong approximation to realize the graph as a quotient of the Bruhat--Tits tree of the homothety classes of $A_\qq$-lattices in $(\mathbb{F}_q(T))_\qq^2$. We also introduce the completeness number $E(\pp)$, the least integer such that the graph is complete for every prime $\qq\neq\pp$ with $\deg\qq\geq E(\pp)$, and derive explicit upper bounds using the Ramanujan--Petersson bound for the eigenvalues of operators associated to Brandt matrices over function fields. In particular, if $d=\deg\pp$, then $E(\pp)\leq 2d+2$, with sharper bounds depending on $q$ and the parity of $d$. This confirms a conjecture of Micheli and Papikian that the graph becomes complete once $\deg\qq$ is sufficiently large relative to $\deg\pp$. We further obtain parity-dependent lower bounds and prove that $E(\pp)\geq d+\log_q d-K$ for an absolute constant $K>0$ and sufficiently large odd $d$, thereby refuting the previously suggested bound $E(\pp)\leq d+1$. Finally, we prove a criterion yielding an algorithm to compute $E(\pp)$.
\end{abstract}

\maketitle

\section{Introduction}

The study of analogies between number fields and function fields is fundamental in number theory. A central example is the similarity between \emph{elliptic curves} and rank-two \emph{Drinfeld modules}.

Over finite fields, elliptic curves can be categorized into \emph{ordinary} and \emph{supersingular}, with supersingular elliptic curves having endomorphism rings that are maximal orders in a quaternion algebra. An \emph{isogeny} is a nonconstant morphism of elliptic curves that preserves the identity point. Fixing a prime $\ell$, one may consider the graph whose vertices are isomorphism classes of supersingular elliptic curves and whose edges correspond to cyclic $\ell$-isogenies, that is, isogenies with cyclic kernel of order $\ell$. These supersingular isogeny graphs were studied, in particular, by Kohel \cite{Kohel1996EndomorphismRO}.

Rank-two Drinfeld modules have corresponding notions of supersingularity and isogeny. Drinfeld modules were introduced by Drinfeld \cite{Drinfeld1974} in his proof of the Langlands conjecture for $\operatorname{GL}(2)$ over function fields. Drinfeld modular curves and Drinfeld modules over finite fields were subsequently studied by Gekeler \cite{Gekeler1986,Gekeler1983,GEKELE1991187}. In particular, supersingular rank-two Drinfeld modules also have endomorphism rings that are maximal orders in quaternion algebras. A general introduction to the theory is given by \cite{papikian_drinfeld_2023}.

Isogenies of Drinfeld modules have also been studied from a computational perspective. Modular polynomials describing cyclic isogenies were considered for Drinfeld modules of arbitrary rank by Breuer and R\"uck \cite{BREUER200959}, while algorithms for computing modular polynomials and isogenies in rank two were developed by Caranay, Greenberg, and Scheidler \cite{caranay-computing-mod-polynomials}.

In the present paper, we study the supersingular isogeny graphs of rank-two Drinfeld modules. For the polynomial ring $A\coloneq\F_q[T]$ in indeterminate $T$, let $\gamma\colon A\to k$ be an $\F_q$-algebra homomorphism to a finite field $k$ for a fixed prime power $q$. A rank-two Drinfeld $A$-module is defined by an additive polynomial
\[\phi_T=\gamma(T)+g\tau+\Delta\tau^2,\qquad \phi_T\in k\{\tau\},\]
where $g$ and $\Delta\neq 0$ are coefficients and $\tau\colon x\mapsto x^q$ denotes the Frobenius endomorphism. See \Cref{sec:prelims} for the complete definitions.

We work with supersingular Drinfeld modules in $A$-characteristic $\pp$, that is, $\pp=\ker(\gamma)$. If a Drinfeld $A$-module $\phi$ is supersingular, that is, its $\pp$-torsion is trivial, then by a theorem of Gekeler \cite[Theorem~4.3]{GEKELE1991187}, the endomorphism ring $\End(\phi)$ is a maximal $A$-order in the quaternion algebra $\End(\phi)\otimes_A\F_q(T)$ ramified precisely at $\pp$ and the infinite place.

It is known that any two supersingular Drinfeld modules of the same rank and characteristic are isogenous \cite[Lemma~4.4.3]{papikian_drinfeld_2023}. A natural question is whether they can be connected using only isogenies associated to a single fixed prime, denoted $\qq$. Specifically, we consider cyclic $\qq$-isogenies, that is, isogenies whose kernels are isomorphic to $A/\qq$. Here and throughout, we use the same symbol for a prime ideal and its generating monic polynomial, that is, $\pp=(\pp)$.

The supersingular $\qq$-isogeny graph $\Gamma_\pp(\qq)$ thus has as its vertices the isomorphism classes of supersingular rank-two Drinfeld $A$-modules in characteristic $\pp$, with edges given by cyclic $\qq$-isogenies. Our first result shows that this graph is always connected:

\begin{theorem}
\label{thm:connectedness}
    Let $A\coloneq\F_q[T]$ and let $\pp$ and $\qq$ be distinct finite primes of $A$. Then the graph $\Gamma_\pp(\qq)$ of supersingular rank-two Drinfeld $A$-modules in characteristic $\pp$, with edges given by cyclic $\qq$-isogenies, is connected.
\end{theorem}

It was conjectured by Micheli and Papikian \cite[Remark~4.2]{micheli2026stabilizationisogenyspacessupersingular} that supersingular isogeny graphs become \emph{complete} (including self-loops) once $\deg\qq$ is large enough relative to $\deg\pp$. Their question arose from the study of rank-metric codes constructed from spaces of morphisms between supersingular Drinfeld modules. More precisely, they studied the \emph{Brandt matrix} $B(\qq)=(b_{i,j}(\qq))$, whose entries $b_{i,j}(\qq)$ count the cyclic $\qq$-isogenies (up to equivalence) between isomorphism classes of supersingular Drinfeld modules.

Accordingly, we introduce the \emph{completeness number} $E(\pp)$ of $\pp$, defined as $$E(\pp)\coloneq\inf\left\{e\in\Z_{\geq1}\colon B(\qq)\text{ has no zero entries for every prime $\qq\neq\pp$ with $\deg\qq\geq e$}\right\},$$ that is, the lowest degree of $\qq$ for which $\Gamma_\pp(\qq)$ is always complete, with $E(\pp)=\infty$ if such a degree does not exist. In the following theorem, we prove that $E(\pp)$ is finite for every $\pp$ and give an explicit upper bound; in particular, this proves the conjecture.

\begin{theorem}
\label{thm:upper-bound-intro}
    Denote $d\coloneq\deg\pp$. If $d=1$ or $d=2$, then $E(\pp)=1$. Suppose that $d\geq3$ and define
    \[C_\pp\coloneq
    \begin{cases}
    \displaystyle
    \frac{q^d-q}{q-1},
    &d\text{ odd},\\[10pt]
    \displaystyle
    \frac{q^d-q^2}{q^2-1},
    &d\text{ even}.
    \end{cases}\]
    Then
    $$E(\pp)\leq\left\lfloor2\log_q\left(C_\pp+\sqrt{C_\pp^2-1}\right)\right\rfloor+1.$$
    In particular, $E(\pp)\leq2d+2$.
\end{theorem}

Micheli and Papikian further suggested that all entries of $B(\qq)$ become positive whenever $\deg\qq>\deg\pp$. The computations in \Cref{app:completeness-algorithm} provide counterexamples to this sharper threshold. We also prove the following lower bounds on $E(\pp)$:

\begin{theorem}
\label{thm:lower-bound}
    Denote $d\coloneq\deg\pp$. If $d\geq 6$ is even, then $E(\pp)\geq d-1$. If $d\geq 3$ is odd, then $E(\pp)\geq d$. Moreover, there exists an absolute constant $K>0$ such that for all sufficiently large odd $d$, we have $E(\pp)\geq d+\log_qd-K$.
\end{theorem}

The paper is structured as follows. In \Cref{sec:prelims}, we introduce the standard notions needed for the proofs of the main results and recall Gekeler's correspondence. In \Cref{sec:connectedness}, we apply Gekeler's correspondence and strong approximation to establish key properties of $\Gamma_\pp(\qq)$ and prove that it is connected. In \Cref{sec:upper-bound}, we recall the Brandt matrix, formally define the completeness number and prove the upper bound on $E(\pp)$. In \Cref{sec:lower-bound}, we prove the lower bounds on $E(\pp)$.
In \Cref{sec:modular-polynomials}, we establish a criterion for completeness based on modular polynomials. In \Cref{app:completeness-algorithm}, we establish a finite algorithm that computes $E(\pp)$ for a given finite prime $\pp$ of $A=\F_q[T]$.

\section{Preliminaries}
\label{sec:prelims}
\subsection{Drinfeld modules, torsion and isogenies}

Fix a prime power $q$. Denote the ring of polynomials in indeterminate $T$ by $A\coloneq\mathbb{F}_q[T]$ and write $F\coloneq\mathbb{F}_q(T)$ for its fraction field. Given a nonzero ideal $\mathfrak n\subseteq A$, we denote by the same symbol the unique monic generator of $\mathfrak n$, that is, $\mathfrak n=(\mathfrak n)$. The \emph{primes} of $A$ are the maximal ideals of $A$; given a prime $\pp$, denote the residue field at $\pp$ by $\F_\pp\coloneq A/\pp$.

Consider a finite field $k$ such that there exists an $\F_q$-algebra homomorphism $\gamma:A\to k$. The \emph{$A$-characteristic} of $k$ is $\ch_A(k)\coloneq\ker(\gamma)$; it is a prime of $A$, which we will denote by $\pp\coloneq\ch_A(k)$ throughout.

A \emph{Drinfeld module} is an action of $A$ on the additive group by additive polynomials. More precisely, if we write $\tau$ for the Frobenius endomorphism $x\mapsto x^q$, then the \emph{twisted polynomial ring} $k\{\tau\}$ is defined by the commutation rule $\tau c=c^q\tau$ for $c\in k$. An element $u=u_0+u_1\tau+\cdots+u_m\tau^m$ (here and throughout, $m$ is called the \emph{degree} of $u$, denoted $\deg(u)$; the lowest $i$ such that $u_i\neq0$ is called the \emph{height} of $u$, denoted $\height(u)$) acts as the additive polynomial
$$x\longrightarrow \operatorname{ev}_xu = u_0x+u_1x^q+\cdots+u_mx^{q^m},$$
defining a ring isomorphism between $k\{\tau\}$ and the ring $k\langle x\rangle$ of $\F_q$-linear polynomials, where multiplication on $k\langle x\rangle$ is defined by composition.

\begin{defn}[Drinfeld module]
\label{defn:drinfeld}
    A \emph{Drinfeld module} of rank $r\geq 1$ over an $A$-field $k$ is an $\F_q$-algebra homomorphism
\begin{align*}
    \phi:A&\longrightarrow k\{\tau\},\\
    a&\longrightarrow\phi_a=\gamma(a)+g_1(a)\tau+\cdots+g_n(a)\tau^n,
\end{align*}
such that for each $a\neq 0$ we have $\deg(a)r=n$ and the coefficient $g_n(a)$ is nonzero.
\end{defn}

\begin{remark}
    Note that $\phi$ is uniquely determined by $\phi_T$, so in order to define a Drinfeld module one may simply choose $g_1,\ldots,g_r\in k$ with $g_r\neq 0$ and set $\phi_T=\gamma(T)+g_1\tau+\cdots+g_r\tau^r$.
\end{remark}

\begin{defn}[Torsion]
\label{defn:torsion}
    Let $\phi$ be a Drinfeld module over a finite $A$-field $k$. If $\mathfrak a\subset A$ is an ideal, define the \emph{$\mathfrak a$-torsion} as
    \[\phi[\mathfrak a]\coloneq\{x\in\bar k\colon\phi_a(x)=0\text{ for every }a\in\mathfrak a\}.\]
    For a principal ideal $(a)$ we write $\phi[a]$ for the kernel of the additive polynomial $\phi_a$.
\end{defn}

A rank-two Drinfeld module $\phi$ of $A$-characteristic $\pp$ is called \emph{supersingular} if its $\pp$-torsion is trivial, that is, if $\phi[\pp]={0}$. Next, we formally define the key maps between Drinfeld modules.

\begin{defn}[Maps between Drinfeld modules]
\label{defn:morphism-group}
    Let $\phi$ and $\psi$ be Drinfeld modules over a finite $A$-field $k$. A \emph{morphism} $u\colon\phi\to\psi$ is an element $u\in k\{\tau\}$ such that $u\phi_a=\psi_au$ for every $a\in A$; equivalently, $u\phi_T=\psi_Tu$. A nonzero morphism $u\colon\phi\to\psi$ is called an \emph{isogeny}. An isogeny is an \emph{isomorphism} if it is invertible in $k$, that is, $u\in k^\times$. The \emph{group of morphisms} between $\phi$ and $\psi$ is $\hom_k(\phi,\psi)$, and the \emph{endomorphism ring} of $\phi$ is $\End_k(\phi)\coloneq\hom_k(\phi,\phi)$. We write $\hom(\phi,\psi)\coloneq\hom_{\bar k}(\phi,\psi)$ for the group of all morphisms $\phi\to\psi$ over the algebraic closure $\bar k$ of $k$. For an isogeny $u\in\hom(\phi,\psi)$, denote by $\ker(u)\coloneq\{\alpha\in\bar k\colon u(\alpha)=0\}$ the set of roots of the polynomial $u(x)\in\bar k\langle x\rangle$. For a finite prime $\qq\neq\ch_A(k)=\pp$, we call $u$ a \emph{cyclic $\qq$-isogeny} if $\ker(u)\cong A/\qq$ as an $A$-module.
\end{defn}


\begin{remark}
    Given $u\in\hom(\phi,\psi)$ and $a\in A$, denote $a\circ u\coloneq u\phi_a=\psi_au$. We have that $a\circ u\in\hom(\phi,\psi)$, so $\hom(\phi,\psi)$ is an $A$-module. The condition $\ker(u)\cong A/\qq$ for $u$ to be a cyclic $\qq$-isogeny is equivalent to $\ker(u)$ being a one-dimensional $A/\qq$-subspace of $\phi[\qq]\cong(A/\qq)^2$.
\end{remark}

For a rank-two Drinfeld module $\phi_T=\gamma(T)+g\tau+\Delta\tau^2$, its \emph{$j$-invariant} is
\[j(\phi)\coloneq\dfrac{g^{q+1}}{\Delta}.\]
It is invariant under isomorphism and, over an algebraically closed field, two rank-two Drinfeld modules are isomorphic if and only if they have the same $j$-invariant.

\subsection{Gekeler's Correspondence}
\label{sec:ideal-class-correspondence}
In the theory of elliptic curves, the \emph{Deuring correspondence} gives a bijection between left ideal classes and isomorphism classes of supersingular elliptic curves \cite{Deuring1941DieTD}. An analogous result was proven by Gekeler \cite{GEKELE1991187}, which we will throughout call \emph{Gekeler's correspondence}.

Let $\phi_0$ be a supersingular rank-two Drinfeld module in characteristic $\pp$. Denote $\End^0(\phi)\coloneq\End(\phi)\otimes_AF$, $\O\coloneq\End(\phi_0)$, and $D\coloneq\End^0(\phi_0)$. By \cite[Theorem 4.3]{GEKELE1991187}, $D$ is the quaternion division algebra over $F$ ramified precisely at $\pp$ and $\infty$, and $\O$ is a maximal $A$-order in $D$.

\begin{defn}[Fractional left ideal, Left ideal class set]
\label{defn:full-left-ideal}
    A \emph{fractional left $\O$-ideal} is a finitely generated $A$-submodule $I\subset D$ such that $\O I\subset I$ and $FI=D$. It is \emph{integral} if $I\subset\O$. Two left ideals $I$ and $J$ are said to be \emph{equivalent} if $J=Ix$ for some $x\in D^\times$. The equivalence class of $I$ is denoted by $[I]$. The set of equivalence classes of fractional left $\O$-ideals is denoted by $\cl(\O)$ and is called the \emph{left ideal class set} of $\O$.
\end{defn}

\begin{defn}[Kernel ideal of an isogeny, equivalent isogenies]
\label{defn:kernel-ideal}
    Let $u\colon\phi_0\to\phi$ be an isogeny. Define the \emph{kernel ideal} of $u$ as $$I(u)\coloneq\hom(\phi,\phi_0)\circ u=\{v\circ u:v\in\hom(\phi,\phi_0)\}\subset\O.$$ Two isogenies are called \emph{equivalent} if they share the same kernel.
\end{defn}

\begin{theorem}[Gekeler's correspondence {\cite[Theorem 4.3(ii)]{GEKELE1991187}}]
\label{thm:gekeler-correspondence}
    The assignment $\phi\longmapsto[I(u)]$, where $u\colon\phi_0\to\phi$ is any isogeny, gives a bijection between the isomorphism classes of supersingular rank-two Drinfeld modules in characteristic $\pp$ and the left ideal classes in $\cl(\O)$.
\end{theorem}

The quaternion algebra $D$ is equipped with a multiplicative map $\nrd\colon D^\times\to F^\times$ known as the \emph{reduced norm}. Note that if $D_v\cong M_2(F_v)$, then the reduced norm is the determinant. For a fractional left $\O$-ideal $I$, define the fractional ideal generated by reduced norms of the elements of $I$ as $$\nrd(I)\coloneq\sum_{x\in I}A\,\nrd(x)\subset F.$$

\begin{defn}[{Norm of an isogeny \cite[\S3.9]{GEKELE1991187}}]
\label{defn:norm-of-isogeny}
    We define the \emph{norm} of an isogeny $u\colon\phi\to\psi$ of Drinfeld $A$-modules of the same rank as $\nn(u)\coloneq\pp^{\height(u)/\deg\pp}\chi(\ker(u))$, where $\chi(M)$ is the \emph{Euler-Poincar\'e characteristic} of a finite $A$-module $M$, that is, an ideal of $A$ determined by
\begin{enumerate}[label=(\roman*)]
    \item $\chi(M)=\rr$, if $M\cong A/\rr$ for a prime ideal $\rr$ of $A$;
    \item If $0\to M_1\to M\to M_2\to0$ is exact, then $\chi(M)=\chi(M_1)\chi(M_2)$.
\end{enumerate}
\end{defn}

\begin{remark}
    If $u\in\End(\phi)$, then $u$ may also be viewed as an element of the quaternion algebra $\End^0(\phi)$, and $\nn(u)=(\nrd(u))$.
\end{remark}

\section{Connectedness via strong approximation}
\label{sec:connectedness}
Under Gekeler's correspondence, the vertices of $\Gamma_\pp(\qq)$ correspond to left ideal classes of $\O$. We first show that cyclic $\qq$-isogenies correspond to $\qq$-neighbor relations between these classes. Locally at $\qq$, the neighbor criterion is adjacency in the Bruhat--Tits tree, which will be introduced in the following subsection. Strong approximation then proves that the ideal-class graph is a quotient of the Bruhat--Tits tree and therefore connected. The corresponding connections among supersingular elliptic-curve isogeny graphs and Bruhat--Tits trees are surveyed in \cite{AmorosEtAl2021}.

Let $v$ be a finite prime of $A$. Denote the completion of $F$ at $v$ by $F_v$ and its valuation ring by $A_v$. Define $D_v\coloneq D\otimes_FF_v$ and $\O_v\coloneq\O\otimes_AA_v$. If $I$ is a fractional left $\O$-ideal, denote its localization at $v$ by $I_v\coloneq I\otimes_AA_v$. Since $\O$ is maximal, every $I_v$ is principal, so there exists some $x_v\in D_v^\times$ such that $I_v=\O_vx_v$.

\subsection{Cyclic isogenies and the Bruhat--Tits tree}
\begin{defn}[$\qq$-neighbor]
\label{defn:q-neighbor}
    Let $I$ and $J$ be fractional left $\O$-ideals. We call $J$ a \emph{$\qq$-neighbor} of $I$ if representatives of their ideal classes may be chosen such that $\qq I\subsetneq J\subsetneq I$ and $I/J\cong(A/\qq)^2$ as $A$-modules.
\end{defn}

\begin{lemma}
\label{lem:nrd-q-neighbor-condition}
    Suppose that $J\subseteq I$ are fractional left $\O$-ideals. Then $J$ is a $\qq$-neighbor of $I$ if and only if $\nrd(J)\nrd(I)^{-1}=\qq$.
\end{lemma}

\begin{proof}
For each finite prime $v$, write $I_v=\O_vx_v$ and $J_v=\O_vy_vx_v$ with $x_v\in D_v^\times$ and $y_v\in\O_v$. Suppose first that $\nrd(J)\nrd(I)^{-1}=\qq$. If $v\neq\qq$, then $\nrd(y_v)\in A_v^\times$, so $y_v\in\O_v^\times$ and $J_v=I_v$.

At $\qq$, set $W\coloneq A_{\qq}$ and let $\varpi$ be a uniformizer. Since $D_{\qq}\cong M_2(F_{\qq})$, identify $\O_{\qq}$ with $M_2(W)$. Smith normal form gives $J_{\qq}=M_2(W)\diag(\varpi^a,\varpi^b)x_\qq$ for some $a,b\geq0$. The relative reduced norm gives $a+b=1$, so $\{a,b\}=\{0,1\}$. Hence $\qq I_\qq\subsetneq J_\qq\subsetneq I_\qq$ and $I_\qq/J_\qq\cong(A/\qq)^2$. Since $J_v=I_v$ for $v\neq\qq$, it follows that $J$ is a $\qq$-neighbor of $I$.

Conversely, suppose that $\qq I_\qq\subsetneq J_\qq\subsetneq I_\qq$ and $I/J\cong(A/\qq)^2$. Then $J_v=I_v$ for $v\neq\qq$. At $\qq$, Smith normal form gives $J_\qq=M_2(W)\diag(\varpi^a,\varpi^b)x_\qq$ with $a,b\in\{0,1\}$. Since $I/J\cong(A/\qq)^2$, localization at $\qq$ gives $\length_W(I_{\qq}/J_{\qq})=2$ (here $\length_W$ denotes module length over the discrete valuation ring $W$, normalized by $\length_W(W/\varpi W)=1$). Moreover, $I_{\qq}/J_{\qq}\cong(W/\varpi^aW)^2\oplus(W/\varpi^bW)^2$, so $\length_W(I_{\qq}/J_{\qq})=2a+2b$; hence, $a+b=1$. Therefore, $\nrd(J)\nrd(I)^{-1}=\qq$.
\end{proof}

\begin{prop}
\label{prop:isogeny-neighbor-correspondence}
    Under Gekeler's correspondence, two left ideal classes are $\qq$-neighbors if and only if the corresponding supersingular isomorphism classes of Drinfeld modules are connected by a cyclic $\qq$-isogeny.
\end{prop}

\begin{proof}
For an integral left $\O$-ideal $K$, let $u_K$ denote the monic generator of the principal left ideal $\bar{k}\{\tau\}K$. Then $u_K$ is an isogeny, its target corresponds to $[K]$, and $\nn(u_K)=\nrd(K)$. Conversely, for an isogeny $w\colon\phi_0\to\eta$, the isogeny associated to $I(w)$ is $w$ up to an isomorphism of its target \cite[proof of Theorem~4.3] {GEKELE1991187}.

Let $v\colon\phi\to\psi$ be a cyclic $\qq$-isogeny and choose an isogeny $u\colon\phi_0\to\phi$. Set $I\coloneq I(u)$ and $J\coloneq I(v\circ u)$. Then
$$J=\hom(\psi,\phi_0)\circ v\circ u\subseteq\hom(\phi,\phi_0)\circ u=I.$$
Since $\qq\neq\pp$, the isogeny $v$ is separable, and therefore $\nn(v)=\chi(\ker v)=\qq$. Hence
$$\nrd(J)\nrd(I)^{-1}=\nn(v\circ u)\nn(u)^{-1}=\qq.$$
By \Cref{lem:nrd-q-neighbor-condition}, the classes of $I$ and $J$ are $\qq$-neighbors.

Conversely, suppose that representatives $J\subseteq I$ are $\qq$-neighbors. After multiplying both ideals by the same nonzero element of $A$, we may assume that they are integral. Since $\bar k\{\tau\}J\subseteq\bar k\{\tau\}I$, there exists a unique monic $v\in\bar k\{\tau\}$ such that $u_J=v\circ u_I$. The defining relations for $u_I$ and $u_J$ imply that $v$ is an isogeny between their targets. By \Cref{lem:nrd-q-neighbor-condition},
$$\nn(v)=\nn(u_J)\nn(u_I)^{-1}=\nrd(J)\nrd(I)^{-1}=\qq.$$
Since $\nn(v)=\qq$ is coprime to the characteristic $\pp$, the isogeny $v$ is separable, and $\chi(\ker v)=\qq$. Since $A$ is a PID and $\qq$ is prime, $\ker v\cong A/\qq$. Thus, $v$ is a cyclic $\qq$-isogeny, concluding the proof.
\end{proof}

Since $D$ is ramified precisely at $\pp$ and $\infty$ and $\qq\neq\pp$, we have $D_\qq\cong M_2(F_\qq)$. After a suitable choice of isomorphism, we may also assume that $\O_\qq\cong M_2(A_\qq)$.

\begin{defn}[$A_\qq$-lattice, Bruhat--Tits tree]
\label{defn:Aq-lattice-Bruhat-Tits-tree}
    An \emph{$A_\qq$-lattice} in $F_\qq^2$ is a free $A_\qq$-submodule $L\subset F_\qq^2$ of rank two such that $F_\qq L=F_\qq^2$. Two lattices $L$ and $L'$ are \emph{homothetic} if $L'=cL$ for some $c\in F_\qq^\times$. The \emph{Bruhat--Tits tree} $\mathcal T_\qq$ is the graph whose vertices are the homothety classes of $A_\qq$-lattices in $F_\qq^2$. Two vertices are adjacent if representatives may be chosen such that $\qq L\subsetneq L'\subsetneq L$; equivalently, $L/L'\cong A/\qq$.
\end{defn}

Each vertex of the $\mathcal T_\qq$ has $|\F_\qq|+1$ neighbors, corresponding to the one-dimensional subspaces of $L/\qq L\cong\F_\qq^2$. Furthermore, the tree $\mathcal T_\qq$ is connected.

A principal left ideal $M_2(A_\qq)b$ corresponds to the $A_\qq$-lattice generated by the rows of $b$. Left multiplication by an element of $M_2(A_\qq)^\times$ changes the row generators but not the lattice. If $L'$ is adjacent to $L$, then the corresponding left ideals satisfy $\qq I\subsetneq J\subsetneq I$ and $I/J\cong(L/L')^2\cong(A/\qq)^2$; hence, adjacency in $\mathcal T_\qq$ is precisely the $\qq$-neighbor relation.

\subsection{Strong approximation and the quotient graph}
\label{sec:strong-approximation}
This subsection follows theory established in \cite[\S28]{VoightQuaternionAlgebras2021}, reformulated to our setting. Define the \emph{norm-one group} of $D$ by $D^1\coloneq\{x\in D^\times\colon\nrd(x)=1\}$. Similarly, define $D_v^1\coloneq\{x\in D_v^\times\colon\nrd(x)=1\}$ and $\O_v^1\coloneq\O_v^\times\cap D_v^1$.

\begin{defn}[Restricted product]
\label{defn:restricted-product}
    Set $S\coloneq\{\infty,\qq\}$. The restricted product of the groups $D_v^1$ away from $S$ is
    $$\widehat D^{1,S}\coloneq\prod_{v\notin S}{}'D_v^1,$$
    where the restricted product is taken with respect to the subgroups $\O_v^1$; that is, an element of $\widehat D^{1,S}$ is a tuple $(x_v)_{v\notin S}$ such that $x_v\in D_v^1$ for every $v\notin S$ and $x_v\in\O_v^1$ for all but finitely many $v$.
\end{defn}

Let $S$ be a set of places. A quaternion algebra $D$ is \emph{$S$-indefinite} if it is split at some place in $S$. In particular, $D_\qq^1\cong\operatorname{SL}_2(F_\qq)$, which is noncompact, hence $D$ is $S$-indefinite.

\begin{theorem}[Strong approximation {\cite[Main Theorem 28.5.3]{VoightQuaternionAlgebras2021}}]
\label{thm:strong-approximation}
    Let $D$ be a quaternion algebra over a global field and suppose $D$ is $S$-indefinite. Then $D^1$ is dense in $\widehat D^{1,S}$.
\end{theorem}

\begin{remark}
    Suppose that for each $v\notin S$ we choose an open subset $U_v\subset D_v^1$, with $U_v=\O_v^1$ for all but finitely many $v$. Then there exists a single element $\gamma\in D^1$ such that $\gamma\in U_v$ for every $v\notin S$. This allows one to impose conditions at finitely many places while requiring $\gamma$ to be integral at all the remaining places.
\end{remark}

\begin{prop}
\label{prop:trivial-away-q}
    Every class in $\cl(\O)$ contains a representative $J$ such that $J_v=\O_v$ for every finite prime $v\neq\qq$.
\end{prop}

\begin{proof}
    Let $I$ be a fractional left $\O$-ideal. Set $R\coloneq A[1/\qq]$, $\O_R\coloneq\O[1/\qq]$, and $L\coloneq I[1/\qq]$. We aim to show that the left $\O_R$-ideal $L$ is principal.

    For each finite prime $v\neq\qq$, choose $x_v\in D_v^\times$ such that $L_v=\O_vx_v$. Define $m_v\coloneq v(\nrd(x_v))$. Observe that only finitely many $m_v$ are nonzero. Moreover, $m_v$ does not depend on the choice of $x_v$: if $x_v'=u_vx_v$ for some $u_v\in\O_v^\times$, then $\nrd(u_v)\in A_v^\times$, so $v(\nrd(x_v'))=v(\nrd(x_v))$.

    Since $A=\F_q[T]$ is a principal ideal domain, its localization $R=A[1/\qq]$ is also a principal ideal domain. Hence there exists $a\in F^\times$ such that $v(a)=m_v$ for every finite prime $v\neq\qq$.

    By the Hasse--Schilling norm theorem \cite[Main Theorem~14.7.4]{VoightQuaternionAlgebras2021},
    $$\nrd(D^\times)=F_{>0,\Omega}^{\times},$$
    where $\Omega$ is the set of real places at which $D$ is ramified. Since $F$ is a global function field, $\Omega=\varnothing$, and hence $\nrd(D^\times)=F^\times$. Choose $\alpha_0\in D^\times$ such that $\nrd(\alpha_0)=a$. Then $v(\nrd(x_v\alpha_0^{-1}))=m_v-v(a)=0$, so $\nrd(x_v\alpha_0^{-1})\in A_v^\times$.

    The reduced norm on the units of a local maximal order is surjective, that is, $\nrd(\O_v^\times)=A_v^\times$ \cite[Lemma 13.4.9]{VoightQuaternionAlgebras2021}. We may therefore choose $u_v\in\O_v^\times$ such that $\nrd(u_v)=\nrd(x_v\alpha_0^{-1})^{-1}$. Set $z_v\coloneq u_vx_v\alpha_0^{-1}$. Then $\nrd(z_v)=1$, so $z_v\in D_v^1$. For all but finitely many $v$, we also have $z_v\in\O_v^1$.

    For each $v\notin S$, define the open subset $U_v\coloneq z_v^{-1}\O_v^1\subset D_v^1$. We have $U_v=\O_v^1$ for all but finitely many $v$. By \Cref{thm:strong-approximation}, there exists a single $\gamma\in D^1$ such that $\gamma\in z_v^{-1}\O_v^1$ for every $v\notin S$. Equivalently, $z_v\gamma\in\O_v^1$ for every finite prime $v\neq\qq$.

    Fix $d\coloneq\alpha_0^{-1}\gamma\in D^\times$. Since $u_vx_vd=z_v\gamma\in\O_v^\times$, we obtain $L_vd=\O_vx_vd=\O_v$ for every finite prime $v\neq\qq$. Therefore $Ld=\O_R$. Finally, set $J\coloneq Id$. The ideals $I$ and $J$ are equivalent, and $J[1/\qq]=I[1/\qq]d=Ld=\O[1/\qq]$. Localizing at any finite prime $v\neq\qq$ gives $J_v=\O_v$. Hence the ideals $I$ and $J$ are equivalent, concluding the proof.
\end{proof}

Define $\Gamma_\qq\coloneq\O[1/\qq]^\times/A[1/\qq]^\times$. Through the fixed isomorphism $D_\qq\cong M_2(F_\qq)$, let $\Gamma_\qq$ act on $\mathcal T_\qq$ by $\alpha\cdot[L]\coloneq[L\alpha^{-1}]$. The denominator consists of the central scalar units, which act trivially on homothety classes, so this action is well-defined.

\begin{prop}
\label{prop:isom-quotient-graph}
    The graph of left ideal classes of $\O$, with edges given by $\qq$-neighbor relations, is isomorphic to $\Gamma_\qq\backslash\mathcal T_\qq$.
\end{prop}

\begin{proof}
    Given $b\in D_\qq^\times$, let $I(b)$ be the fractional left $\O$-ideal defined by $I(b)_v=\O_v$ if $v\neq\qq$ and $I(b)_v=\O_\qq b$ if $v=\qq$. By \Cref{prop:trivial-away-q}, every left ideal class has a representative $I(b)$.

    Suppose that $I(b)$ and $I(b')$ are equivalent. Then $I(b')=I(b)\alpha$ for some $\alpha\in D^\times$. For every finite prime $v\neq\qq$, we have $\O_v=\O_v\alpha$, so $\alpha\in\O_v^\times$; hence, $\alpha\in\O[1/\qq]^\times$. At $\qq$, the equality $\O_\qq b'=\O_\qq b\alpha$ holds if and only if $b'=ub\alpha$ for some $u\in\O_\qq^\times$, which is equivalent to
    $$\cl(\O)\cong\O_\qq^\times\backslash D_\qq^\times/\O[1/\qq]^\times.$$
    On the other hand, the map sending $b$ to the row lattice of $b$ identifies $\O_\qq^\times\backslash D_\qq^\times/F_\qq^\times$ with the vertex set of $\mathcal T_\qq$: left multiplication by $\O_\qq^\times\cong\gl_2(A_\qq)$ changes the choice of basis, while right multiplication by $F_\qq^\times$ rescales the lattice. Since $F_\qq^\times=A_\qq^\times\qq^{\mathbb Z}$, every scalar lies in $\O_\qq^\times A[1/\qq]^\times$. Thus, quotienting by $\O[1/\qq]^\times$ induces on $\mathcal T_\qq$ the action of
    $$\Gamma_\qq\coloneq\O[1/\qq]^\times/A[1/\qq]^\times.$$
    Using the corresponding left action $\alpha\cdot[L]\coloneq[L\alpha^{-1}]$, we obtain $\cl(\O)\cong\Gamma_\qq\backslash\mathcal T_\qq$. Finally, the row-lattice interpretation shows that adjacency in $\mathcal T_\qq$ is precisely the local $\qq$-neighbor relation, and hence the edges of the quotient correspond to $\qq$-neighbor relations between ideal classes.
\end{proof}

\begin{proof}[Proof of \Cref{thm:connectedness}]
    By Gekeler's correspondence (\Cref{thm:gekeler-correspondence}) and \Cref{prop:isogeny-neighbor-correspondence,prop:isom-quotient-graph}, the graph $\Gamma_\pp(\qq)$ is isomorphic to $\Gamma_\qq\backslash\mathcal T_\qq$. Since the Bruhat--Tits tree $\mathcal T_\qq$ is connected, so is its quotient. Thus, $\Gamma_\pp(\qq)$ is connected.
\end{proof}

\section{Spectral Upper Bound on the Completeness Number}
\label{sec:upper-bound}
The Brandt matrix $B(\qq)$ records the number of isogenies between isomorphism classes of supersingular rank-two Drinfeld modules. We use its derived linear operator and the Ramanujan--Petersson theorem over function fields on its nontrivial spectrum to prove that every entry $b_{i,j}(\qq)$ is positive once the degree of the isogeny prime $\qq$ is sufficiently large.

\subsection{Brandt operators and weights}
\begin{defn}[Brandt matrix]
\label{defn:brandt-matrix}
    Let $\phi_1,\ldots,\phi_n$ be the representatives of supersingular Drinfeld $A$-modules of rank $2$ in fixed characteristic $\pp$ up to isomorphism. Let $\mm\in A$ be a monic polynomial. The \emph{Brandt matrix} $B(\mm)=(b_{i,j}(\mm))_{1\leq i,j\leq n}\in\mat_n(\Z)$ is defined by
    $$b_{i,j}(\mm)=\frac{1}{\#\aut(\phi_j)}\#\{u\in\hom(\phi_i,\phi_j)\colon \nn(u)=(\mm)\}.$$
\end{defn}

\begin{remark}
    If $\qq\neq\pp$ is prime, then $b_{i,j}(\qq)$ counts the cyclic $\qq$-isogenies $\phi_i\to\phi_j$, up to postcomposition by automorphisms of $\phi_j$. In particular, $\Gamma_\pp(\qq)$ is complete precisely when $b_{i,j}(\qq)>0$ for every $i,j$ (including $i=j$).
\end{remark}

\begin{defn}[Completeness number]
\label{defn:completeness-number}
For a finite prime $\pp$, define the \emph{completeness number} $$E(\pp)\coloneq\inf\left\{e\in\Z_{\geq1}:b_{i,j}(\qq)>0\text{ for every $i,j$ and every prime $\qq\neq\pp$ with $\deg\qq\geq e$}\right\}.$$
\end{defn}

Thus, $E(\pp)<\infty$ precisely when the graphs $\Gamma_\pp(\qq)$ are complete (including self-loops) for all $\qq$ of sufficiently large degree. \Cref{thm:upper-bound-intro} will show that this always holds.

For a nonzero ideal $\mathfrak a\subset A$, denote $|\mathfrak a|\coloneq\#(A/\mathfrak a)=q^{\deg\mathfrak a}$. Fix a prime $\qq\neq\pp$ and set $d\coloneq\deg\pp$ and $Q\coloneq|\qq|=q^{\deg\qq}$. For each supersingular isomorphism class $\phi_i$, define $w_i\coloneq\#\aut(\phi_i)/(q-1)$. Write
$$M\coloneq\sum_{i=1}^n\frac{1}{w_i}.$$

\begin{lemma}
\label{lem:automorphism-weights}
For every $i$, one has $w_i\in\{1,q+1\}$. If $d$ is odd, precisely one class $\phi_i$ has weight $w_i=q+1$, and if $d$ is even, every class has weight $1$. Furthermore,
$$M=\frac{|\pp|-1}{q^2-1}=\frac{q^d-1}{q^2-1}.$$
\end{lemma}

\begin{proof}
Write $(\phi_i)_T=\gamma(T)+g_i\tau+\Delta_i\tau^2$, where $\Delta_i\neq0$. Every automorphism of $\phi_i$ is an element $c\in\bar{k}^\times$ satisfying $c(\phi_i)_T=(\phi_i)_Tc$. Comparing the coefficients of $\tau$ and $\tau^2$ gives $cg_i=g_ic^q$ and
$$c\Delta_i=\Delta_ic^{q^2}.$$
Since $\Delta_i\neq0$, the second equality implies
$$c^{q^2-1}=1.$$
Thus, $c\in\F_{q^2}^\times$, and if $g_i\neq 0$, then $c\in\F_q^\times$. Hence $w_i=1$, unless $g_i=0$, in which case $w_i=q+1$. The condition $g_i=0$ is equivalent to a zero $j$-invariant
$$j(\phi_i)=g_i^{q+1}/\Delta_i=0,$$
so it determines a unique isomorphism class.

By \cite[Theorem~5.9]{Gekeler1983}, restated in \cite[Theorem~4.1]{PapikianWei2015}, the number $n$ of supersingular rank-two Drinfeld modules in characteristic $\pp$, up to isomorphism, over the algebraic closure, is
\[
n = \begin{dcases}
\frac{|\pp|-1}{q^2-1}, & d \text{ even}, \\[1ex]
\frac{|\pp|-q}{q^2-1}+1, & d \text{ odd},
\end{dcases}
\]
and the unique class with $j$-invariant zero is supersingular if and only if $d$ is odd. The assertions about the weights follow. If $d$ is even, then $M=n$. If $d$ is odd, then
$$M=(n-1)+\frac{1}{q+1}=\frac{|\pp|-q}{q^2-1}+\frac{1}{q+1}=\frac{|\pp|-1}{q^2-1}.\qedhere$$
\end{proof}

\begin{lemma}
\label{lem:brandt-row-sum-duality}
For every fixed $i$,
$$\sum_{j=1}^n b_{i,j}(\qq)=Q+1.$$
Moreover, for each $i$ and $j$ we have $w_jb_{i,j}(\qq)=w_ib_{j,i}(\qq)$.
\end{lemma}

\begin{proof}
The cyclic $\qq$-isogenies with source $\phi_i$, up to postcomposition by automorphisms of the target, are determined by their kernels. Their kernels are the one-dimensional $A/\qq$-subspaces of $\phi_i[\qq]\cong(A/\qq)^2$. The number of such subspaces is $\#\PP^1(A/\qq)=Q+1$. Grouping the isogenies according to the isomorphism class of their targets gives
$$\sum_jb_{i,j}(\qq)=Q+1.$$

Dualization gives a bijection between the isogenies $u\in\hom(\phi_i,\phi_j)$ of norm $\qq$ and $u^\vee\in\hom(\phi_j,\phi_i)$ of norm $\qq$. Therefore, $\#\aut(\phi_j)b_{i,j}(\qq)=\#\aut(\phi_i)b_{j,i}(\qq)$, and division by $q-1$ gives the second identity.
\end{proof}

Let $V$ be an $n$-dimensional complex vector space with basis $e_1,\ldots,e_n$; that is,
$$V\coloneq\bigoplus_{i=1}^n\C e_i.$$
Define a Hermitian inner product on $V$ by
$$ \left\langle\sum_{i=1}^nx_ie_i,\sum_{i=1}^ny_ie_i\right\rangle\coloneq\sum_{i=1}^n\dfrac{x_i\ov{y_i}}{w_i},$$
where $\ov{y_i}$ denotes complex conjugation. Define the \emph{Brandt operator} $T_\qq\colon V\to V$ by
$$T_\qq e_j\coloneq\sum_i b_{i,j}(\qq)e_i.$$
By \Cref{lem:brandt-row-sum-duality},
$$\left\langle T_\qq e_j,e_i\right\rangle=\frac{b_{i,j}(\qq)}{w_i}=\frac{b_{j,i}(\qq)}{w_j}=\left\langle e_j,T_\qq e_i\right\rangle,$$
hence $T_\qq$ is self-adjoint.

Write $\1\coloneq e_1+\cdots+e_n$. The row-sum identity (\Cref{lem:brandt-row-sum-duality}) gives $T_\qq\1=(Q+1)\1$. Define
$$V_0\coloneq (\C\1)^\perp=\left\{\sum_ix_ie_i:\sum_i\dfrac{x_i}{w_i}=0\right\}.$$
Since $T_\qq$ is self-adjoint, it preserves $V_0$.

Denote by
$$\left\lVert T_\qq\big|_{V_0}\right\rVert\coloneq\sup_{\substack{x\in V_0\\x\neq0}}\frac{\|T_\qq x\|}{\|x\|}$$
the operator norm of the restriction of $T_\qq$ to $V_0$.

\begin{prop}[Upper bound on operator norm]
\label{prop:norm-bound}
    For every prime $\qq\neq\pp$, we have $\displaystyle\left\lVert T_\qq\big|_{V_0}\right\rVert\leq 2\sqrt Q$.
\end{prop}

\begin{proof}
    Let $Q+1,\lambda_2(\qq),\ldots,\lambda_n(\qq)$ be the eigenvalues of $T_\qq$, listed with multiplicity, where $Q+1$ is the eigenvalue corresponding to $\1$. Since $T_\qq$ is self-adjoint, by the spectral theorem we obtain an orthonormal basis $v_1,v_2,\ldots,v_n$ of $V$ consisting of eigenvectors of $T_\qq$. We may assume $v_1=\1/\lVert \1\rVert$; then the vectors $v_2,\ldots,v_n$ form an orthonormal basis of $V_0$. Furthermore, $T_\qq v_r=\lambda_r(\qq)v_r$ for each $2\leq r\leq n$.

    For any $$x=\sum^n_{r=2}c_rv_r\in V_0,$$
    we have 
    $$\lVert x\rVert^2=\sum^n_{r=2}\lvert c_r\rvert^2,\qquad\lVert T_\qq x\rVert^2=\sum^n_{r=2}\lvert\lambda_r(\qq)\rvert^2\lvert c_r\rvert^2.$$
    Therefore, $\lVert T_\qq x\rVert\leq\max_{r\geq 2}\lvert\lambda_r(\qq)\rvert\lVert x\rVert$. As derived in \cite[p.~319]{Papikian_2016} from Drinfeld's Ramanujan--Petersson theorem over function fields \cite{Drinfeld1988Petersson}, we have $|\lambda|\leq 2\sqrt Q$. The assertion follows.
\end{proof}

\subsection{An entrywise estimate}
The entries $b_{i,j}$ can be considered as having a main term given by the eigenvalue $Q+1$ and an error term bounded by \Cref{prop:norm-bound}.

\begin{prop}[Entrywise estimate]
\label{prop:entrywise-estimate}
    For every $i,j$, we have
    \[\left|b_{i,j}(\qq)-\frac{Q+1}{w_jM}\right|\leq\frac{2\sqrt Q}{w_jM}\sqrt{(w_iM-1)(w_jM-1)}.\]
\end{prop}

\begin{proof}
    Let $P\colon V\to\C\1$ be the orthogonal projection of $V$ onto $\C\1$. Since $\langle\1,\1\rangle=M$ and $\langle e_j,\1\rangle=1/w_j$, we have
    $$P(e_j)=\frac{1}{w_jM}\1.$$
    Set $e_j^0\coloneq e_j-P(e_j)\in V_0$. Then
    $$T_\qq e_j=\frac{Q+1}{w_jM}\1+T_\qq e_j^0.$$
    Since the coefficient of $e_i$ in a vector $x$ is $w_i\langle x,e_i\rangle$, and $T_\qq e_j^0\in V_0$, comparison of the coefficient of $e_i$ gives
    $$b_{i,j}(\qq)-\frac{Q+1}{w_jM}=w_i\left\langle T_\qq e_j^0,e_i^0\right\rangle.$$
    By \Cref{prop:norm-bound} and the Cauchy--Schwarz inequality,
    \[\left|b_{i,j}(\qq)-\frac{Q+1}{w_jM}\right|\leq2\sqrt{Q}\,w_i\left\lVert e_j^0\right\rVert\left\lVert e_i^0\right\rVert.\]
    Finally, $$\left\lVert e_r^0\right\rVert^2=\frac{1}{w_r}\left(1-\frac{1}{w_rM}\right)=\frac{w_rM-1}{w_r^2M}$$ for every $r$, which gives the inequality.
\end{proof}

Denote $w_{\max}\coloneq\max_iw_i$ and define $C_\pp\coloneq Mw_{\max}-1$. By \Cref{lem:automorphism-weights},
\[
C_\pp = \begin{dcases}
\frac{q^d-q}{q-1}, & d \text{ odd}, \\[1ex]
\frac{q^d-q^2}{q^2-1}, & d \text{ even}.
\end{dcases}
\]
Moreover,
$$\sqrt{(w_iM-1)(w_jM-1)}\leq C_\pp$$
for every $i,j$. By \Cref{prop:entrywise-estimate},
$$b_{i,j}(\qq)\geq\frac{Q+1-2C_\pp\sqrt{Q}}{w_jM}.$$
Therefore, $B(\qq)$ has no zero entries whenever $Q+1>2C_\pp\sqrt Q$.

\subsection{Explicit upper bound}

\begin{proof}[Proof of \Cref{thm:upper-bound-intro}]
If $d=1$ or $d=2$, then $C_\pp=0$, and the preceding estimate shows that every $B(\qq)$ is positive for every prime $\qq\neq\pp$, hence $E(\pp)=1$.

Suppose now that $d\geq 3$. Write
$$x\coloneq\sqrt Q=q^{\deg\qq/2}.$$
The condition $Q+1>2C_\pp\sqrt Q$ is equivalent to $x^2-2C_\pp x+1>0$. The roots of $X^2-2C_\pp X+1$ are
$$C_\pp\pm\sqrt{C_\pp^2-1}.$$
The smaller root is less than $1$, whereas $x>1$. Thus, $B(\qq)$ has no zero entries whenever
$$x>C_\pp+\sqrt{C_\pp^2-1},$$
or equivalently whenever
$$\deg\qq>2\log_q\left(C_\pp+\sqrt{C_\pp^2-1}\right).$$
It follows that
\[E(\pp)\leq\left\lfloor2\log_q\left(C_\pp+\sqrt{C_\pp^2-1}\right)\right\rfloor+1.\]
Finally, $\sqrt{C_\pp^2-1}<C_\pp$ and, in both cases for the parity of $d$, we have
$$C_\pp+\sqrt{C_\pp^2-1}<2C_\pp<\frac{2q^d}{q-1}\leq q^{d+1}.$$
Therefore, $2\log_q\left(C_\pp+\sqrt{C_\pp^2-1}\right)<2d+2$, and hence $E(\pp)\leq 2d+2$.
\end{proof}

\begin{remark}
    Taking the parity of $d$ into consideration, one obtains the slightly sharper bounds $E(\pp)\leq2d+2$ if $d$ is odd and $q=2$, $E(\pp)\leq2d$ if $d$ is odd and $q\geq3$, $E(\pp)\leq2d-1$ if $d$ is even and $q=2$, and $E(\pp)\leq2d-2$ if $d$ is even and $q\geq3$.
\end{remark}

\section{Lower bounds on the completeness number}
\label{sec:lower-bound}
Denote $d\coloneq\deg\pp$. In this section, we prove lower bounds on the completeness number $E(\pp)$. We first obtain elementary parity-dependent bounds by comparing the row sums of the Brandt matrix with the number and weights of the supersingular isomorphism classes. When $d$ is odd, we then obtain a stronger bound by reducing the vanishing of a diagonal Brandt entry to a norm-representation problem over $\F_{q^2}[T]$.
\begin{prop}
\label{prop:lower-bound-easy}
    If $d\geq 6$ is even, then $E(\pp)\geq d-1$. If $d\geq 3$ is odd, then $E(\pp)\geq d$.
\end{prop}

\begin{proof}
    Suppose first that $d$ is even. By \Cref{lem:automorphism-weights}, each weight equals one and the number of vertices is
    $$n=\frac{q^d-1}{q^2-1}=1+q^2+\cdots+q^{d-2}.$$
    Fixing $\deg\qq=d-2$, each row has sum $q^{d-2}+1<n$ for all $d\geq 6$ and therefore has a zero entry.

    Suppose now that $d$ is odd. Index the unique class with weight $q+1$ by $1$. If $B(\qq)$ had no zero entries, then by \Cref{lem:brandt-row-sum-duality} we would have $b_{1,1}(\qq)\geq1$ and $b_{1,j}(\qq)\geq q+1$ for each $j\neq 1$. Since
    $$n-1=\frac{q^d-q}{q^2-1}=q+q^3+\cdots+q^{d-2},$$
    the first row of $B(\qq)$ would have sum at least $1+(n-1)(q+1)=1+q+q^2+\ldots+q^{d-1}$. Fixing $\deg\qq=d-1$, we have a contradiction, since the row sums equal $q^{d-1}+1$.
\end{proof}

For the rest of this section, assume $d$ to be odd. Index the supersingular isomorphism class with $j$-invariant zero by $1$ and normalize it to $(\phi_1)_T=T+\tau^2$. Write $A'\coloneq\F_{q^2}[T]$ and let $\sigma\colon A'\to A'$ be the unique nontrivial $A$-automorphism, given by
$$\sigma\left(\sum_i a_iT^i\right)=\sum_i a_i^qT^i,$$
and define the \emph{norm} of an element $a\in A'$ by
$$N(a)\coloneq a\sigma(a)\in A.$$
Recall that a cyclic $\qq$-isogeny of $\phi_1$ to itself corresponds to an element $u\in\End(\phi_1)$ with $\nrd(u)=\qq$, so writing $u=a+b\tau^d$ reduces its existence to $\qq=N(a)-\pp N(b)$. Denoting  $\pi\coloneq\tau^d$, we have $\End(\phi_1)=A'\oplus A'\pi$ \cite[Lemmas 3.7 and 3.9]{micheli2026stabilizationisogenyspacessupersingular}, hence $\displaystyle\nrd(a+b\pi)=N(a)-\pp N(b)$. For an integer $m\geq0$, define the \emph{norm set} $\N_m$ as the set of monic degree-$m$ polynomials $f\in A$ such that $f=a\sigma(a)$ for some $a\in A'$. A monic polynomial $f\in A$ is a norm from $A'=\F_{q^2}[T]$ if and only if every prime of odd degree occurs in $f$ with even valuation, since the odd-degree primes of $A$ remain inert in $A'$, while the even-degree primes split.

\begin{prop}
\label{prop:equiv-condition-nonzero-1-1}
    Let $\qq\neq\pp$ be a prime of degree $d+2h$, where $h\in\Z_{\geq0}$. Then $b_{1,1}(\qq)>0$ if and only if $\qq-\pp c=N(a)$ for some $c\in\N_{2h}$ and $a\in A'$.
\end{prop}

\begin{proof}
    The condition $b_{1,1}(\qq)>0$ is equivalent to $N(a)-\pp N(b)=\xi\qq$ for some $a,b\in A'$ and $\xi\in\F_q^\times$. Since $N\colon\F_{q^2}^\times\to\F_q^\times$ is surjective, we may assume $\xi=1$ by scaling $a+b\pi$.

    The degrees of $N(a)$ and $N(b)$ are even, whereas $d$ is odd. Thus, the leading term of $N(a)-\pp N(b)=\qq$ is determined by $-\pp N(b)$. Therefore, $\deg b=h$ and $c\coloneq -N(b)$ is monic of degree $2h$. Since $-1$ is a constant norm, $c\in\N_{2h}$, and $\qq-\pp c=N(a)$. Conversely, if $\qq-\pp c=N(a)$ with $c\in\N_{2h}$, choose $b\in A'$ with $N(b)=-c$.
\end{proof}

\begin{lemma}
\label{lem:non-norm}
    For any integer $h\geq 2$, there exist polynomials $M\in A$ and $a_0\in(A/M)^\times$ with $\deg M\ll q^{2h}$ (with an absolute implied constant), such that whenever $f\equiv a_0\pmod M$, the polynomial $f-\pp c$ is not of the form $N(a)$ for any $a\in A'$ and $c\in\N_{2h}$.
\end{lemma}

\begin{proof}
    Let $\ell\coloneq 2h-1$, and let $\P$ denote the set of odd-degree primes $\rr\neq \pp$ with $\deg\rr\leq\ell$. Choose unit classes $a_\rr\bmod \rr^2$, and let $S\subseteq\N_{2h}$ consist of the $c$ for which $\nu_\rr(a_\rr-\pp c)\neq1$ for every $\rr\in\P$. Here and throughout, $\nu_\rr$ denotes the $\rr$-adic valuation.

    Choose the classes $a_{\rr}$ independently and uniformly from $(A/\rr^2)^\times$. For fixed $c\in\mathcal N_{2h}$ and $\rr\in\P$ with $\rr\nmid c$, exactly a proportion $|\rr|^{-1}$ of the choices of $a_{\rr}$ satisfy $\nu_{\rr}(a_{\rr}-\pp c)=1$. If $\rr\mid c$, then $\rr^2\mid c$, since $c$ is a norm and $\deg\rr$ is odd, and hence $\nu_{\rr}(a_{\rr}-\pp c)=0$ for every unit $a_{\rr}$. Thus, writing
    	\[
	\Delta\coloneq
	\prod_{\rr\in\P}
	\left(1-\frac1{|\rr|}\right),
	\]
	we obtain
	\[
	\mathbb E|S|
	=
	\Delta\sum_{c\in\mathcal N_{2h}}
	\prod_{\substack{\rr\in\P\\ \rr\mid c}}
	\left(1-\frac1{|\rr|}\right)^{-1}.
	\]

	For every squarefree $D$ such that
    \[\supp(D)\coloneq\{\rr:\rr\mid D\}\subseteq\P\]
    and $\deg D\le h$, division by $D^2$ gives a bijection
	\[
	\{c\in\mathcal N_{2h}:D^2\mid c\}
	\longrightarrow
	\mathcal N_{2(h-\deg D)}.
	\]
	Therefore,
	\[
	\mathbb E|S|
	=
	\Delta
	\sum_{\substack{D\ {\rm squarefree}\\
	\supp(D)\subseteq\P\\
	\deg D\le h}}
	\frac{\#\mathcal N_{2(h-\deg D)}}
	{\prod_{\rr\mid D}(|\rr|-1)}.
	\]
    By \cite[Theorem~2.1]{Gorodetsky2017},
    \[
    \#\mathcal N_{2m}\ll\frac{q^{2m}}{\sqrt{m+1}}
    \]
    with an absolute implied constant.     Since
    \[
    \frac1{\sqrt{h-\deg D+1}}\le\frac{\sqrt{\deg D+1}}{\sqrt{h+1}}
    \]
    and
    \[
    \sum_D\frac{q^{-2\deg D}\sqrt{\deg D+1}}{\prod_{\rr\mid D}(|\rr|-1)}\ll1,
    \]
    where the latter follows since there are at most $q^m$ monic polynomials $D$ of degree $m$ and $\prod_{\rr\mid D}(|\rr|-1)\geq1$, so that the sum is bounded by
    \[
    \sum_{m\geq0}q^{-m}\sqrt{m+1}\ll1,
    \]
    it follows that
    \[
    \mathbb E|S|\ll\frac{q^{2h}}{\sqrt{h+1}}\Delta.
    \]
    The prime polynomial theorem gives $\Delta\ll(h+1)^{-1/2}$, and consequently
    \[
    \mathbb E|S|\ll\frac{q^{2h}}{h+1}.
    \]
    Hence the classes $a_{\rr}$ may be chosen so that
    \[
    \#S\ll\frac{q^{2h}}{h+1}.
    \]

    By the prime polynomial theorem, for a sufficiently large absolute odd integer $e$ we may choose distinct primes $\ss_c\neq\pp$ of degree $2h+e$, one for each $c\in S$. By the Chinese remainder theorem, there is a unit class $a_0\bmod M$ satisfying $a_0\equiv a_\rr\pmod{\rr^2}$ and $a_0\equiv\pp c+\ss_c\pmod{\ss_c^2}$ for $\rr\in\P$ and $c\in S$, where
    $$M\coloneq\prod_{\rr\in\P}\rr^2\prod_{c\in S}\ss_c^2.$$
    Moreover,
    \[\deg M=2\sum_{\rr\in\P}\deg\rr+2(2h+e)\#S\ll q^{2h},\]
where the first term is bounded using the prime polynomial theorem and
the second by the preceding estimate for $\#S$.

    If $f\equiv a_0\pmod M$, then $\nu_\rr(f-\pp c)=1$ for some $\rr\in\P$ when $c\notin S$, while $\nu_{\ss_c}(f-\pp c)=1$ when $c\in S$. Thus, $f-\pp c$ is not a norm from $A'$ for any $c\in\N_{2h}$.
\end{proof}

\begin{proof}[Proof of \Cref{thm:lower-bound}]
    The lower bounds of $d-1$ when $d$ is even and $d$ when $d$ is odd follow from \Cref{prop:lower-bound-easy}. For the logarithmic improvement when $d$ is odd, let $C_0$ be the absolute constant in \Cref{lem:non-norm} and choose $K_0$ sufficiently large so that
    $$C_02^{-K_0}\leq\frac{1}{4}.$$
    Fix
    $$h\coloneq\left\lfloor\frac{\log_q d-K_0}{2}\right\rfloor.$$
    Then for all sufficiently large odd $d$, \Cref{lem:non-norm} gives a unit class $a_0\bmod M$ with $\deg M\leq d/4$. Writing $n\coloneq d+2h$, the estimate for the number of prime polynomials in arithmetic progressions \cite[(3.3)]{KimMurty2022} gives
    $$\#\{\qq:\deg\qq=n,\ \qq\equiv a_0\!\!\!\pmod M\}\geq\frac{q^n}{n\Phi(M)}-\frac{(\deg M-1)q^{n/2}}{n}>0,$$
    since $\Phi(M)\leq q^{\deg M}$ and $\deg M\leq d/4$. Choose such a prime $\qq$. By \Cref{lem:non-norm}, none of the polynomials $\qq-\pp c$, with $c\in\N_{2h}$, is a norm. Therefore, \Cref{prop:equiv-condition-nonzero-1-1} gives $b_{1,1}(\qq)=0$ and thus
    $$E(\pp)\geq\deg\qq+1=d+2\left\lfloor\frac{\log_q d-K_0}{2}\right\rfloor+1\geq d+\log_q d-K$$
    for a suitable absolute $K>0$.
\end{proof}

\section{Drinfeld modular polynomials}
\label{sec:modular-polynomials}
The upper bound on $E(\pp)$ in \Cref{thm:upper-bound-intro} is obtained using spectral methods for the operator given by the Brandt matrix. To study its sharpness and determine which entries vanish below certain thresholds, we seek an algebraic criterion that may also yield a computational algorithm.

Rank-two Drinfeld modules can be described using modular polynomials \cite{BREUER200959,caranay-computing-mod-polynomials,HSIA1998236,chen2025cmdrinfeldmodulesselfisogenous}, which record the targets of all cyclic $\qq$-isogenies from a fixed vertex, counted with multiplicity. We combine these with the \emph{supersingular polynomial} $S_\pp(X)$, whose roots are the $j$-invariants of the graph vertices, to study all entries of the Brandt matrix (see also \cite{elguindy2013legendredrinfeldmodulesuniversal,hasegawa2017explicitformulasupersingularpolynomial}).

Let $\phi_1,\ldots,\phi_n$ represent the isomorphism classes of supersingular rank-two Drinfeld modules in characteristic $\pp$ over $\ov\F_\pp$, and write $j_i\coloneq j(\phi_i)$ for the $j$-invariant of $\phi_i$. 

\begin{defn}[Supersingular polynomial]
\label{defn:supersingular-polynomial}
    We define the \emph{supersingular polynomial in characteristic $\pp$} by
    $$S_\pp(X)\coloneq\prod_{i=1}^n(X-j_i).$$
\end{defn}

\begin{remark}
    The polynomial $S_\pp$ is squarefree, since $\phi_1,\ldots,\phi_n$ are pairwise nonisomorphic and the $j$-invariant is a complete invariant.
\end{remark}

The set $\{j_1,\ldots,j_n\}$ is stable under $\gal(\ov\F_\pp/\F_\pp)$. Indeed, if $\sigma\in\gal(\ov\F_\pp/\F_\pp)$, applying $\sigma$ to the coefficients of $\phi$ gives a Drinfeld module $\phi^\sigma$ of the same $A$-characteristic, with $j(\phi^\sigma)=\sigma(j(\phi))$. Furthermore, coefficientwise conjugation by $\sigma$ maps $\phi[\pp]$ to $\phi^\sigma[\pp]$, hence supersingularity is preserved. Thus, Galois conjugation permutes the roots $j_1,\ldots,j_n$, and the coefficients of $S_\pp$ lie in $\F_\pp$.

\begin{defn}
\label{defn:modular-polynomial}
    Let $\qq$ be a prime. The \emph{modular $\qq$-polynomial} is the polynomial $\Phi_\qq(X,Y)\in A[X,Y]$, normalized to be monic in $Y$, given by
    $$\Phi_\qq(j(\phi),Y)\coloneq\prod_{\substack{C\subseteq\phi[\qq]\\ C\cong A/\qq}}\bigl(Y-j(\phi/C)\bigr)$$
    for rank-two Drinfeld modules $\phi$ of $A$-characteristic different from $\qq$. We denote the coefficientwise reduction of $\Phi_\qq$ modulo $\pp$ by $\ov\Phi_\qq\in\F_\pp[X,Y]$.
\end{defn}

Fix a prime $\qq\neq\pp$ and write $Q\coloneq\#(A/\qq)$. Since $\phi_i[\qq]\cong(A/\qq)^2$, it contains exactly $Q+1$ cyclic $A/\qq$-submodules. The product in \Cref{defn:modular-polynomial} is indexed by these submodules, so the same target may occur with multiplicity.

\begin{prop}
\label{prop:modular-brandt-factorization}
    For every $1\leq i\leq n$, we have
    $$\ov\Phi_\qq(j_i,Y)=\prod_{j=1}^n(Y-j_j)^{b_{i,j}(\qq)}$$
    in $\ov\F_\pp[Y]$.
\end{prop}

\begin{proof}
    Every quotient $\phi_i/C$ is isomorphic to some $\phi_j$. For a fixed $C$ with $\phi_i/C\cong\phi_j$, the isomorphisms $\phi_i/C\to\phi_j$ form a torsor under $\aut(\phi_j)$. Hence the normalization in the definition of $b_{i,j}(\qq)$ shows that $b_{i,j}(\qq)$ equals the number of cyclic submodules $C\subseteq\phi_i[\qq]$ for which $\phi_i/C\cong\phi_j$. Grouping the factors accordingly concludes the proof.
\end{proof}

We now obtain the main computable criterion.

\begin{theorem}[Polynomial criterion]
\label{thm:polynomial-criterion}
    The graph $\Gamma_\pp(\qq)$ is complete if and only if
    $$\ov\Phi_\qq(X,Y)\in\bigl(S_\pp(X),S_\pp(Y)\bigr)\subseteq\F_\pp[X,Y].$$
\end{theorem}

\begin{proof}
    By \Cref{prop:modular-brandt-factorization}, we have
    $b_{i,j}(\qq)>0$ if and only if $\ov\Phi_\qq(j_i,j_j)=0$. Thus, the graph is complete if and only if $\ov\Phi_\qq$ vanishes on every pair $(j_i,j_j)$.

    Since $S_\pp$ is squarefree, the quotient $$\ov\F_\pp[X,Y]/\bigl(S_\pp(X),S_\pp(Y)\bigr)$$ corresponds, after extending scalars to $\ov\F_\pp$, to the algebra of functions on the set $\{(j_i,j_j):1\leq i,j\leq n\}$; therefore, $\ov\Phi_\qq$ vanishes on every such pair if and only if its image in this quotient is zero.
\end{proof}

Equivalently, reduce $\ov\Phi_\qq(X,Y)$ modulo $S_\pp(X)$ and $S_\pp(Y)$. The graph $\Gamma_\pp(\qq)$ is complete precisely when the resulting bivariate remainder is zero. Thus, we obtain a finite computable algorithm (\Cref{app:completeness-algorithm}).

\section*{Acknowledgments}
This work was started at the 2026 Research Science Institute (RSI), organized by the Center for Excellence in Education and held at the Massachusetts Institute of Technology. The author thanks his mentor, Ayan Nath, for proposing the topic and for his guidance during the research. He is also grateful to Dr. Tanya Khovanova for her advice on the paper and presentation and for the helpful discussions, and to the program supervisors, Prof. David Jerison and Dr. Jonathan Bloom, for their feedback. He thanks AnaMar\'ia Perez for her comments and suggestions on the paper and presentation. The author is grateful to Prof. Bjorn Poonen and Dr. Alexander Petrov for discussions on the direction of this research and for their feedback and advice. He also thanks Dr. Jenny Sendova, Dr. Stanislav Atanasov, Prof. Scott Kominers, Anay Aggarwal, Jaeho Lee, Miroslav Marinov, Miles Edwards, Celine Zhang, Scarlet Gitelson, Timothy Chen, and Michael Luo for their comments and suggestions.

The author thanks MIT, CEE, Sts. Cyril \& Methodius International Foundation, the EVRIKA Foundation, Kaufland Bulgaria, and the Union of Bulgarian Mathematicians for their support. This research was partially supported by the Bulgarian National Program ``Education with Science.'' The author also acknowledges the Bulgarian High School Student Institute of Mathematics and Informatics (HSSIMI) and the High School Student Institute of the BAS (UchI-BAS). He thanks Anatoli Gruncharov for his support. Finally, the author is grateful to his family.

\bibliographystyle{plain}
\bibliography{refs}

\clearpage
\appendix
\section{Algorithm for computing the completeness number}
\label[appendix]{app:completeness-algorithm}

Let $U(\pp)$ denote the upper bound for $E(\pp)$ established in
\Cref{thm:upper-bound-intro}. The polynomial criterion yields the following
algorithm.

\begin{algorithm}
\caption{Computation of $E(\pp)$}
\label{alg:completeness-number}
\begin{algorithmic}[1]
    \Require A prime $\pp\subset A$
    \Ensure The completeness number $E(\pp)$
    \State Compute the supersingular polynomial $S_\pp(X)$
    \State $e\gets0$
    \ForAll{primes $\qq\neq\pp$ with $\deg\qq<U(\pp)$}
        \State Compute the modular $\qq$-polynomial $\Phi_\qq(X,Y)$
        \State Reduce $\Phi_\qq$ modulo $\pp$ to obtain
        $\ov\Phi_\qq(X,Y)$
        \State Compute $R_{\pp,\qq}(X,Y)\gets\ov\Phi_\qq(X,Y)\bmod\bigl(S_\pp(X),S_\pp(Y)\bigr)$
        \If{$R_{\pp,\qq}(X,Y)\neq0$}
            \State $e\gets\max\{e,\deg\qq\}$
        \EndIf
    \EndFor
    \State \Return $e+1$
\end{algorithmic}
\end{algorithm}

\begin{prop}
    \Cref{alg:completeness-number} terminates and returns $E(\pp)$.
\end{prop}

\begin{proof}
There are finitely many primes $\qq$ with $\deg\qq<U(\pp)$. By \Cref{thm:polynomial-criterion}, the remainder $R_{\pp,\qq}$ is nonzero if and only if $\Gamma_\pp(\qq)$ is not complete. Moreover, \Cref{thm:upper-bound-intro} implies that $\Gamma_\pp(\qq)$ is complete whenever $\deg\qq\geq U(\pp)$. Hence $e=\max\{\deg\qq:\Gamma_\pp(\qq)\text{ is not complete}\}$, with $e=0$ if this set is empty. Therefore, the algorithm returns $e+1=E(\pp)$.
\end{proof}

Table~1 gives the data obtained by running the algorithm for several
values of $q$ and $\deg\pp$. In the distribution column, $a\,(m)$ means that exactly $m$ characteristic primes $\pp$ have $E(\pp)=a$.


\begingroup
\centering
\refstepcounter{table}
\label{tab:computed-values}
\textup{Table~\thetable: Exact $E(\pp)$ distributions for $\deg\pp\geq3$.}
\par\smallskip

\small
\setlength{\tabcolsep}{5pt}

\begin{minipage}[t]{0.49\textwidth}
\vspace{0pt}
\centering
\begin{tabular}{@{}c c r l c@{}}
\hline
$q$ & $\deg\pp$ & $\#\pp$ & $E(\pp)$ distribution & $U(\pp)$\\
\hline
$2$ & $3$ & $2$  & $6\ (2)$              & $8$\\
$2$ & $4$ & $3$  & $3\ (1),\ 5\ (2)$     & $6$\\
$2$ & $5$ & $6$  & $10\ (6)$             & $12$\\
$2$ & $6$ & $9$  & $8\ (4),\ 9\ (5)$     & $11$\\
$2$ & $7$ & $18$ & $12\ (16),\ 14\ (2)$  & $16$\\
$2$ & $8$ & $30$ & $11\ (20),\ 12\ (10)$ & $15$\\
$2$ & $9$ & $56$ & $14\ (2),\ 16\ (54)$  & $20$\\
\hline
$3$ & $3$ & $8$  & $4\ (8)$              & $6$\\
$3$ & $4$ & $18$ & $4\ (9),\ 5\ (9)$     & $6$\\
$3$ & $5$ & $48$ & $8\ (48)$             & $10$\\
\hline
\end{tabular}
\end{minipage}
\hfill
\begin{minipage}[t]{0.49\textwidth}
\vspace{0pt}
\centering
\begin{tabular}{@{}c c r l c@{}}
\hline
$q$ & $\deg\pp$ & $\#\pp$ & $E(\pp)$ distribution & $U(\pp)$\\
\hline
$4$  & $3$ & $20$  & $4\ (8),\ 6\ (12)$    & $6$\\
$4$  & $4$ & $60$  & $4\ (12),\ 5\ (48)$   & $5$\\
\hline
$5$  & $3$ & $40$  & $4\ (40)$             & $6$\\
$5$  & $4$ & $150$ & $4\ (130),\ 5\ (20)$  & $5$\\
\hline
$7$  & $3$ & $112$ & $4\ (112)$            & $5$\\
$8$  & $3$ & $168$ & $4\ (168)$            & $5$\\
$11$ & $3$ & $440$ & $4\ (440)$            & $5$\\
\hline
\end{tabular}
\end{minipage}
\par
\endgroup

\subsection*{Data and code}
The implementation of \Cref{alg:completeness-number} is publicly available on GitHub at \url{https://github.com/nikolaveselinov/completeness-number}. Code written with the assistance of Generative AI (Codex 5.6 Sol Ultra, July 2026).

\smallskip
\end{document}